\documentclass[12pt]{amsart}
\usepackage{amsmath,amssymb,multirow,color,placeins,graphicx,realboxes,afterpage,tikz,cite,verbatim}
\usepackage{etoolbox}
\apptocmd{\sloppy}{\hbadness 10000\relax}{}{}

\numberwithin{equation}{section}

\newtheorem{thm}[equation]{Theorem}
\newtheorem{prop}[equation]{Proposition}
\newtheorem{lemma}[equation]{Lemma}
\newtheorem{cor}[equation]{Corollary}

\theoremstyle{definition}
\newtheorem*{rmk}{Remark}

\newtheorem{defn}[equation]{Definition}

\newcommand{\F}{\mathbb{F}}
\newcommand{\bP}{\mathbb{P}}

\newcommand{\G}{\mathbb{G}}

\DeclareMathOperator{\lcm}{lcm}
\DeclareMathOperator{\ord}{ord}

\DeclareMathOperator{\GL}{GL}

\DeclareMathOperator{\PSL}{PSL}
\DeclareMathOperator{\PGL}{PGL}

\DeclareMathOperator{\Aut}{Aut}

\DeclareMathOperator{\Gal}{Gal}

\DeclareMathOperator{\Sym}{Sym}
\renewcommand{\bar}[1]{#1\llap{$\overline{\phantom{\rm#1}}$}}
\newcommand{\abs}[1]{\lvert #1 \rvert}

\usepackage[colorlinks,pagebackref,pdftex, bookmarks=false]{hyperref}

\begin{document}

\title{A new family of exceptional rational functions}

\author{Zhiguo Ding}
\address{
  Hunan Institute of Traffic Engineering,
  Hengyang, Hunan 421001 China
}
\email{ding8191@qq.com}

\author{Michael E. Zieve}
\address{
  Department of Mathematics,
  University of Michigan,
  530 Church Street,
  Ann Arbor, MI 48109-1043 USA
}
\email{zieve@umich.edu}
\urladdr{http://www.math.lsa.umich.edu/$\sim$zieve/}

\date{\today}

\begin{abstract}
For each odd prime power $q$, we construct an infinite sequence 
of rational functions $f(X) \in \F_q(X)$, each of which is \emph{exceptional} in the sense that for infinitely 
many $n$ the map $c \mapsto f(c)$ induces a bijection of $\bP^1(\F_{q^n})$. Moreover, each of our functions 
$f(X)$ is \emph{indecomposable} in the sense that it cannot be written as the composition of lower-degree 
rational functions in $\F_q(X)$. In case $q$ is not a power of $3$, these are the first known examples of indecomposable exceptional rational 
functions $f(X)$ over $\F_q$ which have non-solvable monodromy groups and have arbitrarily large degree.  These are also the first known 
examples of wildly ramified indecomposable exceptional rational functions $f(X)$, other than linear changes 
of polynomials.
\end{abstract}

\thanks{
The second author thanks the National Science Foundation for support under grant DMS-1601844.}

\maketitle


\section{Introduction}

A rational function $f(X) \in \F_q(X)$ is called \emph{exceptional} if the map $c \mapsto f(c)$ is a bijection 
on $\bP^1(\F_{q^n})$ for infinitely many $n$. The study of exceptional rational functions was initiated in Dickson's 
1896 thesis \cite{Di}, which investigated the special case of exceptional polynomials. Recently the paper \cite{ZR} 
introduced a method to produce permutation polynomials over\/ $\F_{q^2}$ of the form $X^i h(X^{q-1})$ from any exceptional 
rational function over\/ $\F_q$; subsequently many authors used parts of this method to produce permutation polynomials 
of this form from some classical families of exceptional rational functions.

Exceptional polynomials have found applications to algebraic geometry \cite{CH,TTV}, inverse Galois theory 
\cite{A1,A2,A3,A4,DM}, combinatorics \cite{Dil,DD}, coding theory \cite{Dil,DobK}, and cryptography \cite{Dil,DD,DobW,DobK}.  
It seems likely that exceptional rational functions will find similar applications.

In this paper we construct a new family of exceptional rational functions. In order to explain how this family relates 
to previously known families, we now recall some known results about exceptional rational functions.

\begin{lemma} \label{comp}
For any $g, h \in \F_q(X)$, the composition $g(h(X))$ is exceptional if and only if both 
$g(X)$ and $h(X)$ are exceptional.
\end{lemma}


\begin{rmk}
The difficult part of Lemma~\ref{comp} is showing that if $g(x)$ permutes $\bP^1(\F_{q^m})$ for infinitely many $m$ and 
$h(x)$ permutes $\bP^1(\F_{q^n})$ for infinitely many $n$ then there are infinitely many $\ell$ for which both $g(x)$ and 
$h(x)$ permute $\bP^1(\F_{q^\ell})$. Lemma~\ref{comp} is an easy consequence of \cite{Co}, and was first stated explicitly 
in \cite{F0}.
\end{rmk}

In light of Lemma~\ref{comp}, it is natural to write exceptional rational functions as compositions of their atomic 
pieces, which are defined as follows.

\begin{defn}
For any field $K$, a rational function $f(X) \in K(X)$ of degree at least $2$ is called \emph{indecomposable} if $f(X)$ 
cannot be written as $g(h(X))$ with $g, h \in K(X)$ each of degree at least $2$.
\end{defn}

In this terminology, Lemma~\ref{comp} has the following immediate consequence.

\begin{cor}\label{ind}
A rational function $f(X) \in \F_q(X)$ of degree at least $2$ is exceptional if and only if it is a composition of 
indecomposable exceptional rational functions.
\end{cor}

Corollary~\ref{ind} plays a crucial role in the theory of exceptional rational functions, since it reduces the study 
of exceptional rational functions to the study of indecomposable exceptional rational functions. The known results 
about indecomposable exceptional \emph{polynomials} are surveyed in \cite{Z}. In short, up to composing on both sides 
with degree-one polynomials, there are seven known families of indecomposable exceptional polynomials, and it is 
expected that there are no others \cite[Rem.~8.4.18]{Z}.  
It is known that any as-yet-unknown indecomposable exceptional polynomial over\/ $\F_q$ must satisfy several restrictive 
properties \cite{FGS,GMZ,GRZ,GZ,Mthesis}: for instance, its degree must be $p^\ell$ where $\ell \ge 4$ and $p$ is the 
characteristic of\/ $\F_q$. We emphasize that by contrast there are thousands of known families of indecomposable 
permutation polynomials.  
Thus, exceptional polynomials are vastly more rare and more difficult to find 
than permutation polynomials.   Moreover,  since there are known methods to construct many classes of permutation polynomials from a single exceptional polynomial, one may view exceptional polynomials as the fundamental objects which are the ultimate reason for the existence of a huge collection of permutation polynomials.  Much less is known about indecomposable exceptional rational functions than polynomials. 
Besides linear changes of variables of polynomials, the only known indecomposable exceptional rational functions 
are R\'edei functions (which are quadratic twists of $X^n$) \cite{R}, certain coordinate projections of elliptic 
curve isogenies \cite{F}, and some examples in small degrees \cite{DZ,GMS,Hou,M0}.

In this paper we present a new family of indecomposable exceptional rational functions. These include rational functions 
of arbitrarily large degree, but they are not members of any of the known families of such functions. The description 
of our rational functions involves the following polynomials.

\begin{defn}
For $a \in \F_q$ and any positive integer $r$, the \emph{Dickson polynomial of the second kind} with degree $r$ and 
parameter $a$ is the unique polynomial $E_r(X,a) \in \F_q[X]$ satisfying the functional equation
\begin{equation} \label{dickson}
E_r\Bigl(X+\frac{a}X,a\Bigr) = \frac{X^{r+1}-(a/X)^{r+1}}{X-a/X};
\end{equation}
explicitly, we have
\begin{equation} \label{dicksoncoeffs}
E_r(X,a) = \sum_{i=0}^{\lfloor r/2\rfloor} \binom{r-i}{i} (-a)^i X^{r-2i}.
\end{equation}
\end{defn}

\begin{rmk}
For proofs that the polynomial in \eqref{dicksoncoeffs} satisfies \eqref{dickson}, see \cite{ACZ,LMT}.
\end{rmk}

We can now state the main result of this paper. Here if $n$ is a positive integer then we write $\ord_2(n)$ for the largest 
integer $i \ge 0$ such that $2^i \mid n$.

\begin{thm} \label{main}
Let $p$ be an odd prime, and let $r = p^{2k}$ and $q = p^{\ell}$, where $k$ and $\ell$ are positive integers with 
$\ord_2(\ell) \le \ord_2(k)$. If $a \in \F_q^*$ is a nonsquare then
\[
f(X) := \frac{E_r(X,a)^{(r+1)/2}}{(X^2-4a)^{(r^2-r)/4}}
\]
is an indecomposable exceptional rational function in\/ $\F_q(X)$ of degree $(r^2+r)/2$, and $f(X)$ permutes\/ $\bP^1(\F_{q^n})$ 
for each odd $n$. Moreover, if $r\ne 9$ then $f(X)$ is indecomposable over the algebraic closure\/ $\bar\F_q$ of\/ $\F_q$.
\end{thm}

\begin{rmk}
When $r=9$, the function $f(X)$ is decomposable over $\F_{q^2}$ since
\[
f(X)=\frac{(X^3+aX+ab)^5}{X^6}\circ \frac{X^3+ab}{X^2+bX+a}\]
when $b^2=a$.
\end{rmk}

The rational functions $f(X)$ in Theorem~\ref{main} are the first known examples of \emph{wildly ramified} indecomposable 
exceptional rational functions which are not linear changes of variables of polynomials. 
%
%
We recall that saying $f(X)$ is wildly ramified means there is some $\alpha \in \bP^1(\bar\F_q)$ for which the local multiplicity 
of $f(X)$ at $\alpha$ is divisible by the characteristic of\/ $\F_q$. Thus, $f(X)$ is wildly ramified since its local multiplicity 
at either square root of $4a$ is $(r^2-r)/4$. The paper \cite{GMS} makes great progress towards a complete classification of all 
indecomposable exceptional rational functions which are not wildly ramified, but little is known about the analogous question for wildly ramified 
rational functions.

Our proof of Theorem~\ref{main} follows the general strategy used in \cite{GZ,LZ}, in that we first determine the Galois groups 
of some function field extensions associated with $f(X)$, and then prove group-theoretic results about these Galois groups which 
imply the properties asserted in Theorem~\ref{main} via the Galois correspondence. In particular, we will find that if $x$ is 
transcendental over\/ $\F_q$ and $f(X)$ is the rational function in Theorem~\ref{main} then the normal closure of the field 
extension $\bar\F_q(x) / \bar\F_q(f(x))$ is a field $\Omega$ such that the field extensions $\Omega / \bar\F_q(f(x))$ and 
$\Omega / \bar\F_q(x)$ are Galois with Galois groups $\PSL_2(r)$ and $D_{(r-1)/2}$, respectively. In particular, the Galois group 
of $\Omega / \bar\F_q(f(x))$ is non-solvable. This Galois group is called the \emph{geometric monodromy group} of $f(X)$.

The rational functions $f(X)$ in Theorem~\ref{main} are the first known examples of indecomposable exceptional rational functions 
$f(X)$ over $\F_q$ with non-solvable geometric monodromy group, where in addition $f(X)$ has arbitrarily large degree and $\F_q$ has arbitrarily large characteristic.  More precisely, the only previously known examples of indecomposable exceptional rational functions with non-solvable geometric monodromy group are some polynomials in characteristics $2$ and $3$ \cite{CM,GRZ,GZ,LZ} and some rational functions in degrees $28$ and $45$
\cite{GMS,M0}.
%
%
%

The importance of the solvability condition has been noted in many settings, since indecomposable rational functions with 
solvable geometric monodromy group are well understood and are known to behave differently from all other rational functions in many settings
\cite{AGMV,Cantat,DKY,DH,Gauthier,GS,GMS,HP,MSS,CMac,CMac2,Milnor,Rittrat,Zariski,Zariski2,Zdunik}. 
%
%
In particular, all known indecomposable exceptional rational functions $f(X) \in \F_q(X)$ with solvable geometric monodromy 
groups and degree not in $\{ 9, 16, 25, 121 \}$ come from a commutative diagram
\FloatBarrier
\begin{figure}[!ht]  
    \centering
    \begin{tikzpicture}[scale=1.5]
        \node (1) at (0,1){$\G_1$};
        \node (2) at (1.9,1) {$\G_2$};
        \node (3) at (1.9,0) {$\bP^1$};
        \node (4) at (0,0) {$\bP^1$};
        \draw[->] (1) to node[above]{$\phi$} (2);
        \draw[->] (2) to node[right]{$\pi_2$} (3);
        \draw[->] (1) to node[left]{$\pi_1$} (4);
        \draw[->] (4) to node[above]{$f$} (3);
    \end{tikzpicture}
\end{figure}
\FloatBarrier
%
\noindent
where $\phi \colon \G_1 \to \G_2$ is a morphism between one-dimensional connected commutative algebraic groups over $\bar\F_q$, 
and $\pi_i \colon \G_i \to \bP^1$ is a Galois map of curves. 
%
%
This yields the classical examples of Dickson polynomials of the first kind $D_n(X,a)$ (which include $X^n = D_n(X,0)$), 
subadditive polynomials, R\'edei functions, and coordinate projections of elliptic curve isogenies (which include Latt\`es maps). 
%
%
Thus, in some sense the solvable examples are familiar and well-understood functions, while non-solvable examples are mysterious and complicated. 

\begin{rmk}
It is known that a rational function $f(X) \in \F_q(X)$ is exceptional if and only if the constant 
multiples of $X-Y$ are the only irreducible polynomials in $\F_q[X,Y]$ which divide the numerator of $f(X)-f(Y)$ and remain 
irreducible in $\bar\F_q[X,Y]$ \cite{Co,GTZ}. Thus if $f(X)$ is the rational function in Theorem~\ref{main} then the numerator 
of $f(X)-f(Y)$ is highly reducible. It seems likely that the factorization of this numerator can be obtained for all such $f(X)$ 
by adapting the method developed in \cite{Zfac}.
\end{rmk}

This paper is organized as follows. In the next section we recall the standard Galois-theoretic setup connecting exceptional 
rational functions with their monodromy groups. Then in Section~\ref{ratfun} we construct various Galois field extensions 
associated to the rational function $f(X)$ in Theorem~\ref{main}.  In light of the known results from Section~\ref{known}, 
the proof of Theorem~\ref{main} is then reduced to showing properties of certain groups of automorphisms of the Galois closure 
of $\F_q(x) / \F_q(f(x))$, which we do in Section~\ref{proof}. Finally, in Section~\ref{ram} we determine the ramification in 
$f(X)$ and the filtrations of higher ramification groups in $\Omega / \bar\F_q(f(x))$, where $\Omega$ is the Galois closure 
of $\bar\F_q(x) / \bar\F_q(f(x))$. This shows in particular that $f(X)$ is not a linear change of variables of a polynomial.
%
%


\section{Monodromy groups}\label{known}

In this section we recall the known Galois-theoretic framework for studying exceptional rational functions. Throughout this 
section, $q$ is a power of a prime $p$, and $x$ is transcendental over\/ $\F_q$.

\begin{defn}
We say that a nonconstant $f(X) \in \F_q(X)$ is \emph{separable} if the field extension\/ $\F_q(x) / \F_q(f(x))$ is separable, 
and we say that $f(X)$ is inseparable otherwise.
\end{defn}

Separable functions have the following simple characterizations \cite[Lemma~2.2]{DZ}:

\begin{lemma}
A nonconstant $f(X) \in \F_q(X)$ is inseparable if and only if its derivative $f'(X)$ equals $0$, or equivalently if and only 
if $f(X) = g(X^p)$ for some $g(X) \in \F_q(X)$.
\end{lemma}

Since plainly the only inseparable indecomposable exceptional rational functions are $\mu \circ X^p$ with $\mu \in \F_q(X)$ 
of degree one, we restrict our attention to the separable case. Here the main tool is as follows.

\begin{defn}
Let $f(X) \in \F_q(X)$ be separable, let $\Omega$ be the Galois closure of\/ $\F_q(x) / \F_q(f(x))$, and let $\F_{q^k}$ be 
the algebraic closure of \/ $\F_q$ in $\Omega$. The \emph{arithmetic monodromy group} of $f(X)$ is $A := \Gal(\Omega/\F_q(f(x)))$, 
viewed as a group of permutations of the set $\Lambda$ of left cosets of $J := \Gal(\Omega/\F_q(x))$ in $A$. The 
\emph{geometric monodromy group} of $f(X)$ is $G := \Gal(\Omega/\F_{q^k}(f(x)))$, viewed as a group of permutations of $\Lambda$.
\end{defn}

These monodromy groups have the following properties.

\begin{lemma}\label{basic}
Let $f(X) \in \F_q(X)$ be separable of degree $n$. Then, in the notation of the above definition,
\begin{enumerate}
\item $A$ and $G$ are transitive subgroups of\/ $\Sym(\Lambda) \cong S_n$, and $G$ is a normal subgroup of $A$ with $A/G$ 
cyclic of order $k$.
\item $f(X)$ is indecomposable over\/ $\F_q$ if and only if $J$ is a maximal proper subgroup of $A$.
\item $f(X)$ is indecomposable over\/ $\bar\F_q$ if and only if the group $H := \Gal(\Omega/\F_{q^k}(x))$ is a maximal proper 
subgroup of $G$.
\item For any field $K$ containing\/ $\F_q$, the field extension $K(x) / K(f(x))$ has degree $n$.
\item $f(X)$ is exceptional if and only if $J$ and $H$ have exactly one common orbit on $\Lambda$.
\end{enumerate}
\end{lemma}

\begin{proof}
This result is classical. For proofs of items (1)--(4) in this setting, see \cite[Lemma~2.4]{DZ}. Item (5) is proved in \cite{Co,GTZ}.
\end{proof}


\section{Construction of $f(X)$}\label{ratfun}

In this section we exhibit the field extensions which connect the rational function $f(X)$ from Theorem~\ref{main} with the framework of the previous section.  Throughout this section we use the following notation:
\begin{itemize}
\item $p$ is an odd prime
\item $r=p^{2k}$ and $q=p^\ell$, where $k,\ell$ are positive integers with $\ord_2(\ell)\le\ord_2(k)$
\item $Q:=p^{\lcm(2k,\ell)}$
\item $u$ is transcendental over $\F_q$
\item $\Omega:=\F_Q(u)$
\item $a\in\F_q^*$ is a nonsquare, and $\sqrt{a}\in\F_{q^2}^*$ is a prescribed square root of $a$
\item $\displaystyle{f(X):=\frac{E_r(X,a)^{(r+1)/2}}{(X^2-4a)^{(r^2-r)/4}}}$
\item $\displaystyle{v:=u^{(r-1)/2}+u^{(1-r)/2}}$
\item $\displaystyle{t:=\frac{(u^{r^2}-u)^{(r+1)/2}}{(u^r-u)^{(r^2+1)/2}}}$
\item $v':=\sqrt{a}v$
\item $t':=\sqrt{a}^r t$.
\end{itemize}

\begin{prop}\label{gal}
The field extensions $\Omega/\F_Q(t')$,\, $\Omega/\F_q(t')$,\, $\Omega/\F_Q(v')$\, and $\Omega/\F_q(v')$\, are all Galois.  Denote their Galois groups by $G,A,H,J$, respectively. Then $\PSL_2(r)\cong G$ via the map sending the matrix $\Bigl[\begin{matrix}b&c\\d&e\end{matrix}\Bigr]$ to the unique automorphism of\/ $\F_Q(u)$ which fixes each element of\/ $\F_Q$ and sends $u\mapsto (bu+c)/(du+e)$.  Also $H$ is dihedral of order $r-1$, and consists of the elements of $G$ mapping $u\mapsto\zeta u^\epsilon$ with $\zeta\in(\F_r^*)^2$ and $\epsilon\in\{1,-1\}$.  For any prescribed nonsquare $\zeta_0\in\F_r^*$, there is a unique automorphism $\sigma$ of $\Omega$ which acts as $q$-th powering on\/ $\F_Q$ and maps $u\mapsto\zeta_0 u$.  This $\sigma$ satisfies $A=\langle \sigma\rangle G$ and $J=\langle \sigma \rangle H$.  Finally, we have $f(v')=t'$.
\end{prop}

\begin{proof}
For any $\Bigl[\begin{matrix}b&c\\d&e\end{matrix}\Bigr]$ in $\GL_2(r)$ there is a unique corresponding automorphism of $\F_Q(u)$ which fixes each element of $\F_Q$ and maps $u\mapsto (bu+c)/(du+e)$.  This correspondence induces an injective homomorphism $\PSL_2(r)\hookrightarrow\Aut(\Omega)$.  Denote the image of this homomorphism by $G$, so that $\PSL_2(r)\cong G$. We know from Galois theory that $\Omega$ is a Galois extension of the fixed field $\Omega^G$, and that the Galois group $\Gal(\Omega/\Omega^G)$ equals $G$.  Dickson \cite[p.~4]{Di2} showed that $\Omega^G=\F_Q(t)$, so that $\Omega/\F_Q(t)$ is Galois with Galois group $G$.  This is easy to verify: $t$ is fixed by each automorphism in $G$ of the form $u\mapsto u+c$ (with $c\in\F_r$) or $u\mapsto bu$ (with $b\in (\F_r^*)^2$ or $u\mapsto 1/u$, and these automorphisms generate $G$ so $G$ fixes $t$; thus $\F_Q(t)\subseteq\Omega^G$, which since $[\Omega:\F_Q(t)]=[\Omega:\Omega^G]$ implies that $\F_Q(t)=\Omega^G$.
%
%

Next consider the subgroup $H$ of $G$ consisting of the elements mapping $u\mapsto\zeta u^\epsilon$ with $\zeta\in(\F_r^*)^2$ and $\epsilon\in\{1,-1\}$.  Plainly $H$ is dihedral of order $r-1$.  Since $\Omega/\Omega^H$ is Galois with Galois group $H$, we have $[\Omega:\Omega^H]=r-1$.  Since each element of $H$ fixes $v$, we have $\F_Q(v)\subseteq\Omega^H$.  
Here $[\Omega:\F_Q(v)]$ equals the degree of $v$ as a rational function in $u$ (by Lemma~\ref{basic}), which is $r-1$.  Thus $[\Omega:\Omega^H]=[\Omega:\F_Q(v)]$, so since $\F_Q(v)\subseteq\Omega^H$ it follows that $\Omega^H=\F_Q(v)$, whence $\Omega/\F_Q(v)$ is Galois with Galois group $H$.

Since $\sqrt{a}\in\F_{q^2}\subseteq\F_Q$, we have $\F_Q(t')=\F_Q(t)$ and $\F_Q(v')=\F_Q(v)$, so that $\Omega/\F_Q(t')$ and $\Omega/\F_Q(v')$ are Galois with Galois groups $G$ and $H$, respectively.
Let $\sigma$ be the automorphism of $\F_Q(u)$ which acts as $q$-th powering on $\F_Q$ and maps $u\mapsto\zeta_0 u$ for some prescribed nonsquare $\zeta_0\in\F_r^*$.
Then $\sigma$ acts as multiplication by $-1$ on each of $\sqrt{a}$, $u^{(r-1)/2}$, $v$, and $t$, so that $\sigma$ fixes $v'$ and $t'$.  Since $\sigma$ fixes $\F_q$, it follows that $\sigma$ fixes $\F_q(v')$ and $\F_q(t')$.
The group $G$ fixes $\F_Q$, while the restriction of $\sigma$ to $\F_Q$ has order $[\F_Q:\F_q]$, so the set-theoretic product $\langle\sigma\rangle G$ has order at least
\[
\abs{G}\cdot[\F_Q:\F_q]=[\Omega:\F_Q(t')]\cdot [\F_Q(t'):\F_q(t')]=[\Omega:\F_q(t')].
\]
Since $A:=\langle \sigma,G\rangle$ fixes $\F_q(t')$, we have $\Omega^A\supseteq\F_q(t')$, and since $[\Omega:\F_q(t')]<\infty$ it follows that
\[
[\Omega:\F_q(t')]\ge [\Omega:\Omega^A]=\abs{A}.
\]
Thus
\[
\abs{\langle\sigma\rangle G}\ge\abs{G}\cdot[\F_Q:\F_q]\ge \abs{A},
\]
so since $A\supseteq\langle\sigma\rangle G$ we conclude that $A=\langle\sigma\rangle G$ and that $\F_q(t')=\Omega^A$, whence $\Omega/\F_q(t')$ is Galois with Galois group $A$.
%
%
Likewise $\Omega/\F_q(v')$ is Galois with Galois group $J:=\langle H,\sigma\rangle$, and $J=\langle \sigma\rangle H$.

It remains only to show that $f(\sqrt{a}v)=\sqrt{a}^r t$.
Writing $w:=u^{(r-1)/2}$, we have
\[
t=\frac{(w^{2r+2}-1)^{(r+1)/2}}{w^r(w^2-1)^{(r^2+1)/2}}
\]
and
\[
v=w+w^{-1}.
\]
Thus
\[
E_r(v,1) = \frac{w^{r+1}-w^{-r-1}}{w-w^{-1}}=\frac{w^{2r+2}-1}{w^r(w^2-1)}
\]
and
\[
v^2-4=(w-w^{-1})^2
\]
so that
\[
f_0(X):=\frac{E_r(X,1)^{(r+1)/2}}{(X^2-4)^{(r^2-r)/4}}
\]
satisfies
\[
f_0(v)=t.
\]
We have
\begin{equation}\label{eq1}
E_r\Bigl(\frac X{\sqrt{a}},1\Bigr)=\frac{E_r(X,a)}{\sqrt{a}^{r}}
\end{equation}
and
\begin{equation}\label{eq2}
\Bigl(\frac{X}{\sqrt{a}}\Bigr)^2-4=\frac1a (X^2-4a).
\end{equation}
Dividing the $(r+1)/2$-th power of \eqref{eq1} by the $(r^2-r)/4$-th power of \eqref{eq2} yields
\[
f_0\Bigl(\frac X{\sqrt{a}}\Bigr)=f(X)\sqrt{a}^{(r^2-r)/2-(r^2+r)/2}=\frac{f(X)}{\sqrt{a}^{r}},
\]
so that
\[
f(v')=\sqrt{a}^r f_0(v)=\sqrt{a}^r t = t'. \qedhere
\]
\end{proof}


\section{Proof of Theorem~\ref{main}}\label{proof}

In this section we prove Theorem~\ref{main}.  
We use the notation from Proposition~\ref{gal} and in the list preceding that result.  We begin by proving the following properties of the Galois groups appearing in Proposition~\ref{gal}:

\begin{prop}\label{gp}
The following are true:
\renewcommand{\theenumi}{\ref{gp}.\arabic{enumi}}
\renewcommand{\labelenumi}{(\thethm.\arabic{enumi})}
\begin{enumerate}
\item\label{p5} if $r\ne 9$ then there are no groups strictly between $H$ and $G$
\item\label{p6} $J$ does not contain any nontrivial normal subgroups of $A$
\item\label{p7} $H$ and $J$ have a unique common orbit in their action by left multiplication on the set $A/J$ of left cosets of $J$ in $A$.
\end{enumerate}
\end{prop}
\renewcommand{\theenumi}{\arabic{enumi}}
\renewcommand{\labelenumi}{(\arabic{enumi})}

\begin{rmk}
The result \eqref{p7} can be deduced from \cite[Ex.~3.19]{GMS}, whose proof is suitable for expert group theorists.  We give a more elementary and self-contained proof for the benefit of a wider range of readers.
\end{rmk}

\begin{proof}[Proof of Proposition~\ref{gp}]
We first prove \eqref{p5}.
The subgroups of the group $\PSL_2(r)$ were determined around 1900 by Moore \cite{Moore} and Wiman \cite{Wiman} (see also \cite[\S 260]{Di3} or \cite[\S 3.6]{Suzuki}). Inspection of the list of these groups shows that if $r\ne 9$ then each proper subgroup of $\PSL_2(r)$ with order a proper multiple of $r-1$ is isomorphic to $\PGL_2(\sqrt{r})$.
%
%
If $r\ne 9$ and $\PGL_2(\sqrt{r})$ had a subgroup $T$ isomorphic to $D_{(r-1)/2}$ then $T':=T\cap \PSL_2(\sqrt{r})$ would have index at most $2$ in $T$; but by the Moore--Wiman result $\PSL_2(\sqrt{r})$ has no cyclic or dihedral subgroups of order divisible by $(r-1)/2$, so no such $T'$ exists, which proves \eqref{p5}.
%

We now prove \eqref{p6}.  Note that each element of $J$ maps $u\mapsto\zeta u^e$ with $\zeta\in\F_r^*$ and $e\in\{1,-1\}$, and also acts on $\F_Q$ as $q^i$-th powering for some $i\ge 0$.  Let $N$ be a normal subgroup of $A$ which is contained in $J$, and let $n$ be any element of $N$, so that $n(u)=\zeta u^e$ and $n\mid_{\F_Q}$ is the $q^i$-th power map.  For any $w\in \F_r^*$, let $\theta$ be the unique element of $G$ which maps $u\mapsto u+w$.  Then $\rho:=\theta^{-1}n\theta$ maps $u\mapsto \zeta (u-w)^e+w^{q^i}$.  Since $\rho$ is in $N$, and hence is in $J$, we must have $e=1$, and also $\zeta w=w^{q^i}$ for each $w\in\F_r^*$, whence $\zeta=1$ and the $q^i$-th power map acts as the identity on $\F_r$, and hence also on $\F_Q$.  Thus $n$ is the identity element of $N$, so $N=1$, which proves \eqref{p6}.

It remains to prove \eqref{p7}.  The coset $J$ in $A/J$ is a fixed point in the actions of both $J$ and $H$ on $A/J$, so we must show that no other coset lies in a common orbit of $J$ and $H$.
Pick $\theta\in A\setminus J$, and suppose that the $J$-orbit of the coset $\theta J\in A/J$ is also an $H$-orbit.  Since $J=\langle\sigma\rangle H$, this says that the sets $H\theta J$ and $\sigma H\theta J$ are identical, so that $\sigma \theta\in H\theta J$.  Thus $\sigma \theta\in H\theta S$ where $S$ is the set of elements of $J$ whose restriction to $\F_Q$ equals that of $\theta^{-1}\sigma\theta$, or equivalently that of $\sigma$.  It follows that $S=\sigma H$, so that $\sigma \theta\in H\theta\sigma H$.  Write $\theta(u):=(bu+c)/(du+e)$ and suppose $\theta$ acts on $\F_Q$ as the $q^i$-th power map. 
Note that $H$ consists of the elements in $G$ which preserve $\{0,\infty\}$; since $\sigma$ preserves $\{0,\infty\}$, the facts that $A=\langle\sigma\rangle G$ and $J=\langle\sigma\rangle H$ implies that $J$ consists of the elements in $G$ which preserve $\{0,\infty\}$.  Thus since $\theta\notin J$ we see that at most one of $b,c,d,e$ equals zero.  Then
$\theta\sigma(u)=\zeta_0^{q^i}(bu+c)/(du+e)$, so that
$\sigma\theta(u)=(b^q\zeta_0u+c^q)/(d^q\zeta_0u+e^q)$ equals one of the following for some $\zeta,\zeta'\in (\F_r^*)^2$:
\begin{gather}
\label{1}
\zeta^{q^{i+1}} \zeta_0^{q^i}\frac{b\zeta' u+c}{d\zeta' u+e} \\
\label{2}
\zeta^{q^{i+1}} \zeta_0^{q^i}\frac{b\zeta'+cu}{d\zeta'+eu} \\
\label{3}
\zeta^{q^{i+1}} \zeta_0^{-q^i}\frac{d\zeta' u+e}{b\zeta' u+c} \\
\label{4}
\zeta^{q^{i+1}} \zeta_0^{-q^i}\frac{d\zeta'+eu}{b\zeta'+cu}.
\end{gather}
Note that $\zeta^{q^{i+1}}\zeta_0^{q_i}$ and $\zeta^{q^{i+1}}\zeta_0^{-q^i}$ are nonsquares in $\F_r^*$.  First suppose that $\sigma\theta(u)$ equals \eqref{1} or \eqref{3}.  Then $\sigma\theta(\infty)=b^q/d^q$ equals a nonsquare in $\F_r^*$ times either $b/d$ or $d/b$, so we must have $0\in\{b,d\}$.
Likewise, comparing values at $0$ shows $0\in\{c,e\}$, contradicting the condition that at most one of $b,c,d,e$ equals zero.
Next suppose that $\sigma\theta(u)$ equals \eqref{2}.
Comparing values at $\infty$ yields $(b/d)^q=\zeta^{q^{i+1}}\zeta_0^{q^i}c/e$, so since at least three of $b,c,d,e$ are nonzero it follows that all four of them are nonzero.  The ratio of the two coefficients of the numerator of $\sigma\theta(u)$ is $(b/c)^q\zeta_0$, and the corresponding ratio in \eqref{2} is $c/(b\zeta')$, so that $(c/b)^{q+1}=\zeta'\zeta_0$, which is impossible since the left side is a square in $\F_r^*$ while the right side is not.  Finally, suppose that $\sigma\theta(u)$ equals \eqref{4}.  Equating values at infinity yields
\begin{equation}\label{use}
\Bigl(\frac bd\Bigr)^q= \zeta^{q^{i+1}}\zeta_0^{-q^i}\frac ec.
\end{equation}  Since at most one of $b,c,d,e$ equals zero, it follows that none of $b,c,d,e$ are zero, so the ratio between the coefficients of the numerator of $\sigma\theta(u)$ is $(b/c)^q\zeta_0$, which equals $e/(d\zeta')$.  Equating ratios between coefficients of denominators yields $(d/e)^q\zeta_0=c/(b\zeta')$.  Thus
\[
\Bigl(\frac bc\Bigr)^q\frac de=\frac1{\zeta_0\zeta'}=\Bigl(\frac de\Bigr)^q\frac bc,
\]
so that
\[
\Bigl(\frac bc\Bigr)^{q-1}=\Bigl(\frac de\Bigr)^{q-1},
\]
whence
\[
\frac bc=\gamma \frac de
\]
for some $\gamma\in\F_q^*$.  Since also $\gamma\in\F_r^*$, we have $\gamma\in\F_q\cap\F_r\subseteq\F_{\sqrt{r}}$, so that $\gamma$ is a square in $\F_r^*$.  But \eqref{use} shows that the nonsquare $\zeta^{q^{i+1}}\zeta_0^{-q^i}$ in $\F_r^*$ equals
\[
\Bigl(\frac bd\Bigr)^q\cdot\frac ce = \gamma\cdot \Bigl(\frac{b^{(q-1)/2}c}{d^{(q+1)/2}}\Bigr)^2,
\]
which is a contradiction since the right side is a square in $\F_r^*$.
\end{proof}

We now prove Theorem~\ref{main}.

\begin{proof}[Proof of Theorem~\ref{main}]
By the Galois correspondence, \eqref{p6} implies that $\Omega$ is the Galois closure of the field extension $\Omega^J/\Omega^A$, which equals $\F_q(v')/\F_q(t')$.  Since $v'$ is transcendental over $\F_q$ and $t'=f(v')$, this shows that $A$ is the arithmetic monodromy group of $f(X)$, and also that $\deg(f)=[A:J]=(r^2+r)/2$.  Moreover, since $\F_Q$ is the algebraic closure of $\F_q$ in $\Omega=\F_Q(u)$, the geometric monodromy group of $f(X)$ is $\Gal(\Omega/\F_Q(t'))=G$.
By Lemma~\ref{basic}, if $r\ne 9$ then \eqref{p5} implies that $f(X)$ is indecomposable over $\bar\F_q$ (and hence over $\F_q$).  If $r=9$ then it is easy to verify directly that $J$ is maximal in $A$, so that $f(X)$ is indecomposable over $\F_q$.
Likewise, Lemma 2.4 and \eqref{p7} imply that $f(X)$ is exceptional over $\F_q$.  Finally, if $n$ is an odd positive integer then $a$ is a nonsquare in $\F_{q^n}^*$ and $\ord_2(\ell n)=\ord_2(\ell)\le\ord_2(k)$, so the hypotheses of Proposition~\ref{gal} remain valid if we replace $\ell$, $q$, and $Q$ by $\ell n$, $q^n$, and $p^{\lcm(2k,\ell n)}$, respectively.  Thus $f(X)$ is exceptional over $\F_{q^n}$, so that $f(X)$ permutes $\bP^1(\F_{q^{mn}})$ for infinitely many $m$, which implies that $f(X)$ permutes $\bP^1(\F_{q^n})$.
\end{proof}


\section{The ramification of $f(X)$}
\label{ram}

In this section we determine the ramification of the rational function $f(X)$ from Theorem~\ref{main}, which shows in particular that $f(X)$ is not a linear change of variables of a polynomial.  We use the notation from Proposition~\ref{gal} and in the list preceding that result.

\begin{prop}\label{pr}
Each $d\in \bP^1(\bar\F_q)\setminus\{0,\infty\}$ has $(r^2+r)/2$ preimages under $f(X)$, each of which is unramified.  The set $f^{-1}(\infty)$ has size $3$, and contains one point with ramification index $r$ and two with ramification index $(r^2-r)/4$.  The set $f^{-1}(0)$ has size $r$, and all its points have ramification index $(r+1)/2$.  The inertia group of $u=c$ relative to the extension $\bar\F_q(u)/\bar\F_q(t')$ has order
\[
\begin{array}{c@{\mathrel{}{}}l@{}}
1 &\qquad\text{ if }c\in\bar\F_r\setminus\F_{r^2} \\[\smallskipamount]
\frac{r+1}2&\qquad\text{ if }c\in \F_{r^2}\setminus\F_r \\[\smallskipamount]
\frac{r^2-r}2& \qquad\text{ if }c\in\bP^1(\F_r).\smallskip
\end{array}
\]
\noindent
This group is cyclic in the second case, and in the third case is a Borel subgroup of\/ $\PSL_2(q)$ (i.e., it is conjugate to the group of upper-triangular matrices).
For each $c$, the second ramification group in the lower numbering is trivial, and the first ramification group is the unique Sylow $p$-subgroup of the inertia group.
\end{prop}

\begin{proof}
Recall that $t'=g(u)$ where
\[
g(X) = \sqrt{a}^r \frac{(X^{r^2}-X)^{(r+1)/2}}{(X^r-X)^{(r^2+1)/2}}.
\]
Since 
\[
\gcd((X^{r^2}-X)^{(r+1)/2},(X^r-X)^{(r^2+1)/2})=(X^r-X)^{(r+1)/2},
\]
it follows that the $g$-preimages of $\infty$ are the elements of $\bP^1(\F_r)$, each of which has ramification index $(r^2-r)/2$, and likewise the $g$-preimages of $0$ are the elements of $\F_{r^2}\setminus\F_r$, each of which has ramification index $(r+1)/2$.  If $g(X)$ ramified at an element $\alpha\in\bar\F_r\setminus\F_{r^2}$ then $\alpha$ would be a zero of
\begin{align*}
\sqrt{a}^{-r} g'(X)(X^r-X)^{r^2+1} =& -\frac{r+1}2 (X^{r^2}-X)^{(r-1)/2} (X^r-X)^{(r^2+1)/2} \\
&+ \frac{r^2+1}2 (X^r-X)^{(r^2-1)/2} (X^{r^2}-X)^{(r+1)/2} \\
&\!\!\!\!\!\!\!\!\!\!\!\! = \frac{r+1}2 (X^{r^2}-X)^{(r-1)/2}(X^r-X)^{(r^2-1)/2} h(X)
\end{align*}
where
\[
h(X):=-(X^r-X) +  (X^{r^2}-X) = X^{r^2}-X^r = (X^r-X)^r;
\]
but this is impossible since $\alpha\notin\F_{r^2}$.  Thus all ramification in $\bar\F_q(u)/\bar\F_q(t')$ occurs over $t'=0$ or $t'=\infty$.  Hence also all ramification in $\bar\F_q(v')/\bar\F_q(t')$ occurs over $t'=0$ or $t'=\infty$.  The equation \eqref{dickson} shows that all zeroes of $E_r(X+aX^{-1},a)$ have multiplicity $1$, so that also all zeroes of $E_r(X,a)$ have multiplicity $1$.  Also $X^2-4a$ is squarefree, and the polynomials $E_r(X,a)$ and $X^2-4a$ are coprime since $\deg(f)=(r^2+r)/2$ (or alternately, since substituting $X+a/X$ into these two polynomials yields rational functions which visibly have no common zeroes).
Thus the ramification of $f(X)$ is as described in Proposition~\ref{pr}.

We now compute the inertia groups and higher ramification groups in $\bar\F_q(u)/\bar\F_q(t')$.  It suffices to do this at one point of $\bar\F_q(u)$ lying over each of $t'=0$ and $t'=\infty$, since the filtrations of ramification groups at any two points of $\bar\F_q(u)$ lying over the same point of $\bar\F_q(t')$ are conjugate under $G$.
The inertia group at $u=\infty$ consists of the maps $u\mapsto bu+c$ with $b\in\F_r^*$ and $c\in\F_r$.  The first ramification group at $u=\infty$ is the (unique) Sylow $p$-subgroup of this group, which consists of the maps $u\mapsto u+c$ with $c\in\F_r$.  A uniformizer for $u=\infty$ is $1/u$, and if $\rho\colon u\mapsto u+c$ with $c\in\F_r^*$ then \[
\rho\Bigl(\frac1u\Bigr)-\frac1u=\frac1{u+c}-\frac1u=\frac{-c}{u^2+cu}=\frac{-c}{u^2(1+\frac cu)}
\]
vanishes to order $2$ at $u=\infty$, whence $\rho$ is not in the second ramification group at $u=\infty$.  Finally, for any $\alpha\in\F_{r^2}\setminus\F_r$ the ramification index of $u=\alpha$ in $\bar\F_q(u)/\bar\F_q(t')$ is $(r+1)/2$, and hence is coprime to $p$, so the inertia group is cyclic of order $(r+1)/2$ and the first ramification group is trivial.
\end{proof}

\begin{rmk}
The fact that the second ramification groups are trivial also
follows from the combination of the Riemann--Hurwitz genus formula and Hilbert's different formula applied to the extension $\bar\F_q(u)/\bar\F_q(t')$.  This conclusion holds true more generally for finite Galois extensions $L/K$ of algebraic function fields in characteristic $p$ in which the degree-$0$ part of the divisor class group of $L$ has a subgroup isomorphic to $(\mathbb Z/p\mathbb Z)^g$, where $g$ denotes the genus of $L$; cf.\ \cite[Lemma~2.4]{GRZ}.
\end{rmk}



\end{document}